\documentclass[a4paper,11pt,UKenglish]{scrartcl}
\usepackage[affil-it]{authblk}

\usepackage[backend=biber,style=alphabetic, giveninits=true, 
doi=false,isbn=false,url=false, maxbibnames=8, maxalphanames=3, 
hyperref]{biblatex}
\makeatletter
\DeclareOldFontCommand{\rm}{\normalfont\rmfamily}{\mathrm}
\DeclareOldFontCommand{\sf}{\normalfont\sffamily}{\mathsf}
\DeclareOldFontCommand{\tt}{\normalfont\ttfamily}{\mathtt}
\DeclareOldFontCommand{\bf}{\normalfont\bfseries}{\mathbf}
\DeclareOldFontCommand{\it}{\normalfont\itshape}{\mathit}
\DeclareOldFontCommand{\sl}{\normalfont\slshape}{\@nomath\sl}
\DeclareOldFontCommand{\sc}{\normalfont\scshape}{\@nomath\sc}
\makeatother
\setkomafont{disposition}{\rmfamily\bfseries}
\setkomafont{descriptionlabel}{\rmfamily\bfseries}
\usepackage[UKenglish]{babel}

\usepackage{latexsym}
\usepackage[all]{xy}

\usepackage{tikz}
\usepackage{amssymb,amsxtra} 
\usepackage{amsmath, mathtools} 
\usepackage{amsthm}
\usepackage{geometry}
\usepackage{enumerate}
\usepackage[colorinlistoftodos,shadow]{todonotes}

\usepackage[colorlinks]{hyperref}

\def\N{{\mathbb{N}}}
\def\Z{{\mathbb{Z}}}
\def\Q{{\mathcal{Q}}}
\def\K{{\mathbb{K}}}
\def\A{{\mathcal{A}}}
\def\B{{\mathcal{B}}}

\newcommand{\Am}{(\mathcal{A},m)}

\newcommand{\p}{\partial}

\DeclareMathOperator{\pdeg}{pdeg}

\DeclareMathOperator{\codim}{codim}

\DeclareMathOperator{\Der}{Der}

\numberwithin{equation}{section}

\theoremstyle{break}
\newtheorem{theorem}{Theorem}[section]
\newtheorem{prop}[theorem]{Proposition}
\newtheorem{cor}[theorem]{Corollary}
\newtheorem{lemma}[theorem]{Lemma}
\newtheorem{definition}[theorem]{Definition}

\newtheorem{example}[theorem]{Example}

\title{Locally Heavy Hyperplanes in Multiarrangements}
\author{Takuro Abe\footnote{Institute of Mathematics for Industry, Kyushu University, Japan. Email: abe@imi.kyushu-u.ac.jp} and 
Lukas K\"{u}hne\footnote{Einstein Institute of Mathematics, The Hebrew University of Jerusalem, Israel and Max Planck Institute for Mathematics in the Sciences in Leipzig, Germany. Email: lf.kuehne@gmail.com}
}
\date{\today} 

\begin{document}

\maketitle

\begin{abstract}
Hyperplane Arrangements of rank $3$ admitting an unbalanced Ziegler restriction are known to fulfill Terao's conjecture.
This long-standing conjecture asks whether the freeness of an arrangement is determined by its combinatorics.
In this note we prove that arrangements which admit a locally heavy flag satisfy Terao's conjecture which is a generalization of the statement above to arbitrary dimension.
To this end we extend results characterizing the freeness of multiarrangements with a heavy hyperplane to those satisfying the weaker notion of a locally heavy hyperplane.
As a corollary we give a new proof that irreducible arrangements with a generic hyperplane are totally nonfree.
In another application we show that an irreducible multiarrangement of rank $3$ with at least two locally heavy hyperplanes is not free.
\end{abstract}



\section{Introduction}
Inspired by singularity theory, K.\ Saito initiated the study of logarithmic vector fields on a hypersurface~\cite{Sai80}.
In the special case of hyperplane arrangements H.\ Terao subsequently showed that on can pass from analytic to algebraic consideration by introducing free arrangements~\cite{Ter80}.
The long-standing open problem in this area is Terao's conjecture asserting that the freeness of an hyperplane arrangement only depends on its underlying combinatorics.
A recent approach to this problem, which gives a partial answer, is based on multiarrangements and Yoshinaga's criterion, cf.~\cite{Yos04,Yos05,AY13,Abe16}.

In~\cite{AK18}, the authors defined heavy and locally heavy hyperplanes in multiarrangements and developed the theory of the former.
As a main result, heavy flags are introduced as a class of arrangements for which Terao's conjecture holds.
The goal of this note is to generalize this theory to locally heavy hyperplanes.
We start by giving our main definition of (locally) heavy hyperplanes.

\begin{definition}
Let $(\A,m)$ be a multiarrangement in $V=\K^\ell$. 
\begin{itemize}
\item[(1)]
A hyperplane $H_0\in\A $ is called \textbf{heavy} if
\[m(H_0) \ge \sum_{L\in\A, \  L\neq H_0} m(L).\]
\item[(2)]
A hyperplane $H_0\in\A $ is called \textbf{locally heavy} if 
\begin{equation}\label{eq:local_heavy}
m(H_0) \ge \sum_{L\in\A_X, \  L\neq H_0} m(L)
\end{equation}
for all localizations $(\A_X,m_X)$ with $X \in \A^{H_0}$ and 
$|\A_X| \ge 3$.

\end{itemize}
\end{definition}

An advantage of working with locally heavy hyperplanes in multiarrangements is that the Euler restriction $(\A^H,m^*)$ onto a locally heavy hyperplane $H_0$ is of a particular nice shape, namely $m^*(X) = |m_X| - m(H_0)$ for all $X\in \A^{H_0}$.
Since this last equation always holds for $X\in \A^{H_0}$ with $|\A_X|=2$ we only demand Equation~\eqref{eq:local_heavy} to hold for the flats $X \in \A^{H_0}$  with $|\A_X|\ge 3$.

In particular, the Euler multiplicity onto a locally heavy hyperplane is combinatorially determined which is usually not the case, since its definition is of algebraic nature (cf.~\cite{ATW08}).
This multiplicity on the restriction $\A^H$ also coincides with the natural generalization of Ziegler's multirestriction to the setting of multiarrangements.
We call the restriction to a locally heavy hyperplane therefore an \textbf{Euler--Ziegler restriction}.

An important tool to investigate the freeness of an arrangements is its \textbf{second Betti number} $b_2$.
This approach has been pioneered by Yoshinaga in~\cite{Yos05} and by the first author in joint work with Yoshinaga in~\cite{AY13}.
In the case of a simple complex arrangement, this number agrees with the second topological Betti number of the complement with rational coefficients.
In general, the second Betti number can be defined as follows:

\begin{definition}\label{def:b2}
	\begin{enumerate}
		\item[(1)] For a simple arrangement $\A$ its second Betti number $b_2(\A)$ is
		\[
			b_2(\A) \coloneqq \sum_{X\in L_2(\A)} |\A_X|-1.
		\]
		\item[(2)]  For a multiarrangement $\Am$ its second Betti number $b_2\Am$ is defined as
		\[
				b_2\Am\coloneqq \sum_{X\in L_2(\A)} d_X^{(1)}d_X^{(2)},
		\]
		where $d_X^{(1)},d_X^{(2)}$ are the nonzero exponents of the multiarrangement $(\A_X,m_X)$ which is a multiarrangement of rank $2$ and thus free.
	\end{enumerate}
\end{definition}
For a multiarrangement $\Am$ of rank $\ell$, the Betti number $b_2\Am$ is the coefficient of $t^{\ell-2}$ of the characteristic polynomial defined in~\cite{ATW07}.
Note that in the case of a multiarrangement, these coefficients were not called Betti numbers in~\cite{ATW07}; we introduced this term for multiarrangements in~\cite{AK18}.

A main result of this article is a strengthening of Theorem 1.2 in ~\cite{AK18} which establishes a connection between the freeness of a locally heavy multiarrangement and its Euler--Ziegler restriction:

\begin{theorem}\label{th:loc_heavy}
	Let $(\A,m)$ be a multiarrangement with a locally heavy hyperplane $H_0 \in \A$ with 
	$m_0 \coloneqq m(H_0)$. Then 
	\begin{equation}\label{eq:b2}
	b_2(\A,m)-m_0(|m|-m_0) \ge b_2(\A^{H_0},m^{H_0}),
	\end{equation}
	Moreover, 
	$(\A,m)$ is free 
	if and only if the Euler--Ziegler restriction $(\A^{H_0},m^{H_0})$ of 
	$(\A,m)$ onto $H_0$ is free, and Inequality~\eqref{eq:b2} holds with equality.
\end{theorem}

The left hand side of Inequality~\eqref{eq:b2} is the second Betti number of the multiarrangement $\Am$ away from the locally heavy hyperplane $H_0$.
This statement will be made precise in Lemma~\ref{lem:b2}.

The following example is an application of Theorem~\ref{th:loc_heavy} to a multiarrangement whose underlying simple arrangement is the braid arrangement $A_3$.
\begin{example}\label{ex:1}
Consider the multiarrangement $\Am$ given by the defining equation
\[
\Q\Am = x^a(x-y)^a(x-z)^ay^a(y-z)^az^{m_0},
\]
for some positive integers $a$ and $m_0$ with $m_0\ge 2a$.
\begin{figure}
\centering
\begin{tikzpicture}
[is/.style={circle,draw=black!100,fill=white!100},
la/.style={circle},scale=1.1 ,rotate=-45
]	
	\draw (1,-2) -- (-2,1) -- (-2,2) -- (-1,2) -- (2,-1);
	\draw (-1,-2) -- (2,1) -- (2,2) -- (1,2) -- (-2,-1);
	\draw (0,-2.25) -- (0,{sqrt(8)});
	\draw (0,0) circle [radius={sqrt(8)}];
	
	\node at (1,0) [is] {$a^2$};
	\node at (0,1) [is] {$k$};
	\node at (-1,0) [is] {$a^2$};
	\node at (0,-1) [is] {$k$};
	\node at (-1.2,-2.2) [la] {$a$};
	\node at (-2.2,-1.2) [la] {$a$};
	\node at ( 0,-2.6) [la] {$a$};
	\node at (1.2,-2.2) [la] {$a$};
	\node at (2.2,-1.2) [la] {$a$};
	\node at (2.2,-2.2) [la] {$m_0$};		
	
\end{tikzpicture}
\caption{The projectivized picture of $\Am$. The numbers in the intersections denote the value of $b_2$ in the localization along these flats where $k \coloneqq \left\lfloor\frac{3a}{2}\right\rfloor\left\lceil\frac{3a}{2}\right\rceil$.}\label{fig:delA3}
\end{figure}
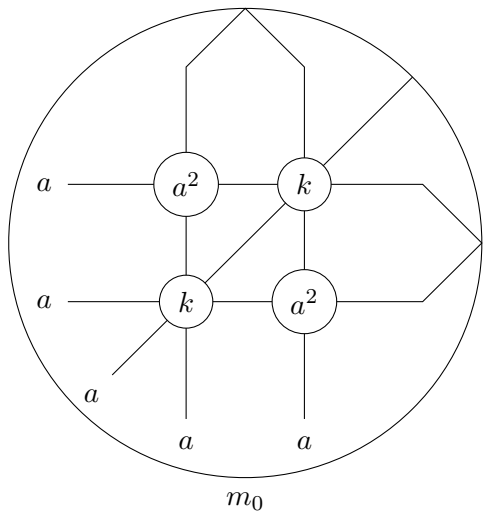
Figure~\ref{fig:delA3} shows a projectivized picture of $\Am$ where $H_0\coloneqq \{z=0\}$ is the hyperplane at infinity.
Note that by assumption $H_0$ is locally heavy.

In the circled intersections at rank a rank two flat $X$ in Figure~\ref{fig:delA3} we display the Betti number $b_2(\A_X,m_X)$.
If the intersection is a rank two flat of size two this is just the product of the multiplicities of  the two defining hyperplanes.
If the intersection is defined by three hyperplanes we use Wakamiko's result in~\cite{Wak07} that yields $b_2(\A_X,m_X)=\left\lfloor\frac{m_X}{2}\right\rfloor\left\lceil\frac{m_X}{2}\right\rceil$ in this case (assuming that this localization does not have a heavy hyperplane which is true in all applications of this theorem in this example).

Now, we can compute the left hand side of Inequality~\eqref{eq:b2} using the above discussion and Definition~\ref{def:b2} as
\begin{equation}\label{eq:lhs}
2a^2 +2\left\lfloor\frac{3a}{2}\right\rfloor\left\lceil\frac{3a}{2}\right\rceil.
\end{equation}
On the other hand, the defining equation $\Q(\A^{H_0},m^{H_0})$ of the Euler--Ziegler restriction to $H_0$ is $x^{2a}y^{2a}(x-y)^a$.
This restricted multiarrangement is of rank $2$ and therefore free.
Wakamiko's theorem~\cite{Wak07} again yields $\exp(\A^{H_0},m^{H_0})=\left(\left\lfloor\frac{5a}{2}\right\rfloor,\left\lceil\frac{5a}{2}\right\rceil\right)$ and therefore
\begin{equation}\label{eq:rhs}
b_2(\A^{H_0},m^{H_0})=\left\lfloor\frac{5a}{2}\right\rfloor\left\lceil\frac{5a}{2}\right\rceil.
\end{equation}
Hence, Theorem~\ref{th:loc_heavy} shows that $\Am$ is free if and only if the Expressions~\eqref{eq:lhs} and~\eqref{eq:rhs} agree.
A short computations yields that this is the case if and only if $a=1$ in which case $\Am$ is free with $\exp\Am=(m_0,2,3)$.
\end{example}

One of the most important open conjectures in the field of hyperplane arrangements is Terao's conjecture. It asks whether the freeness of an arrangement $\A$ is determined by its intersection lattice $L(\A)$.
In this note we will consider Terao's conjecture for the special class of arrangements which iteratively admit locally heavy Euler--Ziegler restrictions.
We call such a chain of restrictions a locally heavy flag:

\begin{definition}
	Let $\A$ be an arrangement in $V=\K^\ell$.
	A flag $\{X_i\}_{i=1}^\ell$ of $\A$ with $X_i \in L_i(\A)$, i.e.\ $\codim X_i = i$, is called a \textbf{locally heavy flag} if $X_{i+1} \in \A^{X_i}$ is locally heavy in $(\A^{X_i},m^{X_i})$ for 
	$i=1,\ldots,\ell-1$.
	Let $LHF_\ell$ be the set of hyperplane arrangements in 
	$V=\K^\ell$ such that 
	every $\A \in LHF_\ell$ admits a locally heavy flag.
\end{definition}

Theorem~\ref{th:loc_heavy} yields a positive answer to Terao's conjecture for the class of arrangements in $LHF_\ell$ which generalizes Corollary 1.8 in~\cite{AK18}:

\begin{theorem}\label{th:flag}
	Let $\A$ be an arrangement in $V=\K^\ell$ which admits a locally heavy flag $\{X_i\}_{i=1}^\ell$.
	Then $\A$ is free with
	$\exp(\A)=(1,m^{X_1}(X_{2}),\ldots,
	m^{X_{\ell-1}}(X_{\ell}))$ if and only if 
	\begin{equation}\label{flagEq}
	b_2(\A)=\sum_{0 \le i< j \le \ell-1}m^{X_i}(X_{i+1}) m^{X_j}(X_{j+1}),
	\end{equation}
	where we set $m^{X_0}(X_1)\coloneqq 1$.
	In this case, this flag is both a
	supersolvable filtration and a divisional flag as in \cite{Abe16}.
	In particular, the freeness of any
	$\A \in LHF_\ell$ 
	depends only on its combinatorics, namely its intersection lattice $L(\A)$.
\end{theorem}

We remark that not every supersolvable arrangement admits a locally heavy flag; for instance the simple braid arrangement $A_3$ underlying Example~\ref{ex:1} is supersolvable but does not have a locally heavy Euler--Ziegler restriction.

As an illustration of Theorem~\ref{th:flag}, we give an example of an arrangement with a locally heavy flag in $LHF_4$ and determine its freeness combinatorially:
\begin{example}
Let $\A$ in $V=\K^4$ be defined by
\[
x(x-z+w)(x-z-w)(y-w)(y+w)(y-z)z(z+w)(z-w)w=0,
\]
Then, $\A$ has a locally heavy flag $X_1\coloneqq\{w=0\} \in \A$, $X_2\coloneqq\{w=z=0\} \in (\A^{X_1},m^{X_1})$ and $X_3\coloneqq\{w=z=y=0\} \in (\A^{X_2},m^{X_2})$, and $X_4\coloneqq 0$.
Then it holds that $m^{X_1}(X_2)=3$, $m^{X_2}(X_3)=3$ and $m^{X_3}(X_4)=3$.
So the RHS of~\eqref{flagEq} evaluates to $1(3+3+3)+3(3+3)+3\cdot3=36$.

We compute $b_2(\A)$ via Definition~\ref{def:b2}.
There are $21$ rank two flats with $|\A_X|=2$, six with $|\A_X|=3$, one with $|\A_X|=4$.
Thus, we obtain $b_2(\A)=21+6\cdot 2+3=36$.
Therefore, Theorem~\ref{th:flag} implies that $\A$ is free and any arrangement with the same intersection lattice as $\A$ is free as well.
\end{example}

This article is organized as follows.
In Section~\ref{sec:pre}, we recall definitions and results relevant for the proofs of the main results.
These proofs will be given in Section~\ref{sec:proof}.
Subsequently, we will consider two applications of the locally heaviness techniques.
Firstly, we will show in Section~\ref{sec:generic} that arrangements with a generic hyperplane are totally nonfree arrangements.
Secondly, we will consider multiarrangements with multiple locally heavy hyperplanes in Section~\ref{sec:multiple}.

\textbf{Acknowledgments}
We thank the anonymous referee for carefully reading an earlier version of this article and for giving many useful suggestions how to improve the exposition.
The first author is partially supported by JSPS KAKENHI Grant-in-Aid for Scientific Research (B) 16H03924.
The second author is supported by ERC StG 716424 - CASe and a Minerva Fellowship of the Max Planck Society.

\section{Preliminaries}\label{sec:pre}

In this section we fix some notations and introduce known results, which will be used in the following proofs.

Let $V$ be a vector space of dimension $\ell$ over a field $\K$ and $S\coloneqq S(V^*)$ be the symmetric algebra. We can choose coordinates ${x_1,\ldots,x_\ell}$ for $V^*$ such that $S=\K[x_1,\ldots,x_\ell]$. A hyperplane $H$ in $V$ is a linear subspace of codimension 1. A (central) \textbf{arrangement of hyperplanes} $\A$ is a finite collection of hyperplanes. This article is mainly concerned with \textbf{multiarrangements}, which are defined to be an arrangement of hyperplanes $\A$ with a multiplicity function $m : \A \rightarrow \Z_{>0}$. Multiarrangements were first defined by Ziegler in~\cite{Zie89} and are denoted by $\Am$.
Define $|m|\coloneqq \sum_{H\in\A} m(H)$. A multiarrangement $\Am$ with $m(H)=1$ for all $H\in\A$ is called a simple hyperplane arrangement.
For each hyperplane $H$ we can choose a linear defining equation $\alpha_H \in S$. Then $\Q \Am \coloneqq \prod_{H\in \A}\alpha_H^{m(H)}$ is the \textbf{defining polynomial} of a multiarrangement $\Am$.
For $L\in\A$ we define the \textbf{characteristic multiplicity} $\delta_L : \A \rightarrow \Z_{>0}$ of $L$ by setting $\delta_L(H)=1$ if $H=L$ and 0 otherwise.

For an arrangement $\A$ the set of all non-empty intersections of elements of $\A$ is defined to be the \textbf{intersection lattice} $L(\A)$, i.e., 
$$
L(\A)\coloneqq\{\cap_{H \in \B} H \mid \B \subset \A\}.
$$
It is ordered by reverse inclusion and ranked by the codimension. 
Denote by $L_r(\A)\coloneqq \lbrace X\in L(\A) \mid \codim(X)=r \rbrace$ the set of $X\in L(\A)$ with codimension $r$.
For any $X\subset V$ let $(\A_X,m_X)$ be the \textbf{localization} of $\Am$ at $X$ by defining $\A_X \coloneqq \lbrace  H \in \A \mid X\subseteq H\rbrace$ and $m_X\coloneqq m|_{\A_X}$.

A central notion of this article is the freeness of a (multi-)arrangement.
The $S$-module $D\Am$ is the \textbf{module of logarithmic derivations} of $\Am$ defined as
\[D \Am \coloneqq\lbrace \theta\in \Der(S) \mid \theta(\alpha_H) \in \alpha_H^{m(H)} S  \mbox{ for all }H\in \A \rbrace,\]
where $\Der(S)$ is the module of all derivations on $S$.
If $D\Am$ is a free $S$-module, we call $\Am$ a \textbf{free multiarrangement}. 
In the case of a free multiarrangement $\Am$ one can choose a homogeneous basis $\theta_1 , \ldots , \theta_\ell$ of $D\Am$.
In this case we define $\exp \Am =(\pdeg \theta_1 , \ldots, \pdeg \theta_\ell )$ to be the \textbf{exponents} of $\Am$ where a derivation $\theta\in \Der(S)$ is homogeneous with $\pdeg \theta = d$ if $\theta(\alpha)$ is a homogeneous polynomial of degree $d$ for any $\alpha\in V^*$.

A useful result to decide the freeness of a multiarrangement is \textbf{Saito's criterion} first proved by K. Saito for simple arrangements~\cite[Theorem 1.8]{Sai80} and by Ziegler for multiarrangements~\cite[Theorem 8]{Zie89}.
For $\theta_1,...,\theta_\ell\in D\Am$ define an $(\ell \times \ell)$-matrix $M(\theta_1,...,\theta_\ell)$ by setting the $(i,j)$-th entry to be $\theta_j(x_i)$.

\begin{theorem}\label{thm:saito}
Let  $\theta_1,...,\theta_\ell$ be derivations in $D\Am$. Then there exists some $f\in S$ such that
\[ \det M(\theta_1,...,\theta_\ell) = f \Q\Am.\] 
Furthermore, the family $(\theta_1,...,\theta_\ell)$ with $\theta_i\in D\Am$ forms a basis of $D\Am$ if and only if $f \in \K^* $.
In particular, if $\theta_1,...,\theta_\ell$ are all homogeneous, then $(\theta_1,...,\theta_\ell)$ forms a basis of $D\Am$ if and only if the following two conditions are satisfied:
\begin{itemize}
\item[(i)] $\theta_1,...,\theta_\ell$ are independent over $S$.
\item[(ii)] $\sum_{i=1}^\ell \pdeg \theta_i =|m|$.
\end{itemize}  
\end{theorem}

Next, we review the addition-deletion theorem for multiarrangements. Let $\Am$ be a multiarrangement and $H_0\in\A$ a fixed hyperplane. 
The deletion $(\A',m')$ of $\Am$ with respect to $H_0$ is defined as:

\begin{enumerate}
\item[(1)] If $m(H_0)=1$ then $\A'\coloneqq\A\setminus \lbrace H_0 \rbrace$ and $m'(H)=m(H)$ for all $H\in\A'$.
\item[(2)] If $m(H_0)\geq 2$ then $\A'\coloneqq\A$ and for $H\in\A'=\A$, we set
\[m'(H)\coloneqq
\left\{
	\begin{array}{ll}
		m(H)  & \mbox{if }  H \neq H_0,\\
		m(H_0) -1  & \mbox{if } H = H_0.
	\end{array}
\right.
\]
\end{enumerate}
The restricted arrangement is defined by $\A^{H_0}\coloneqq \lbrace H_0 \cap H\mid H\in \A\setminus \lbrace H_0\rbrace \rbrace$. 
Now for $X\in \A^{H_0}$, the arrangement $(\A_X,m_X)$ is of rank $2$ and therefore always free (cf.\ \cite[Corollary 7]{Zie89}).
Hence, we can choose a basis $\lbrace \zeta_1,\ldots,\zeta_{\ell -2}, \theta_X, \psi_X \rbrace$ of $D(\A_X,m_X)$ such that $\pdeg \zeta_i =0$, $\theta_X \not\in \alpha_{H_0} \Der_\K (S)$ and $\psi_X \in \alpha_{H_0} \Der_\K (S)$ (cf.\ \cite[Proposition 2.1]{ATW08}).
Then the \textbf{Euler multiplicity} $m^*$ 
on $\A^{H_0}$ is defined as $m^*(X) \coloneqq \pdeg \theta_X$ and $(\A^{H_0},m^*)$ is called the \textbf{Euler restriction} of $\Am$.
For the Euler restriction, one has the following result.
\begin{theorem}\cite[Theorem 0.8]{ATW08}\label{thm:add_del}\textbf{(Addition-Deletion Theorem for multiarrangements)}
Let $\Am$ be a multiarrangement and $H_0\in\A$.
Then any two of the following three statements imply the third:
\begin{enumerate}[(1)]
\item\label{6.5a} $\Am$ is free with $\exp\Am=(d_1,...,d_{\ell-1},d_\ell)$.
\item\label{6.5b} $(\A',m')$ is free with $\exp(\A',m') =(d_1,...,d_{\ell-1},d_\ell-1)$.
\item\label{6.5c} $(\A^{H_0},m^*)$ is free with $\exp(\A^{H_0},m^*)=(d_1,...,d_{\ell-1})$.
\end{enumerate}
\end{theorem}

Another way of defining a multiplicity on the restriction of a multiarrangement $\Am$ to a hyperplane $H_0$ is the \textbf{Euler--Ziegler restriction} $(\A^{H_0},m^{H_0})$ defined by setting
\[
m^{H_0}(X) = |m_X| - m(H_0)
\] 
for any $X\in \A^{H_0}$ (cf. also~\cite{AK18}).
For a simple arrangement $\A$ this definition agrees with the classical \textbf{Ziegler restriction} (cf.~\cite[Example 2]{Zie89}) which we therefore take as its definition for this article.

If $H_0$ is a locally heavy hyperplane in $\Am$, we have $m^*(X)=m^{H_0}(X)$ by~\cite[Proposition 4.1]{ATW08}.
 Hence in the locally heavy case, the Euler restriction $(\A^{H_0},m^*)$ coincides with the Euler--Ziegler restriction $(\A^{H_0},m^{H_0})$.

Much of our work on locally heavy multiarrangements is motivated by and modeled after results on simple arrangements.
The corresponding result to Theorem~\ref{th:loc_heavy} is the following.
\begin{theorem}\cite[Theorem 5.1]{AY13}\label{thm:AY}
Let $\A$ be an arrangement with a fixed hyperplane $H_0\in \A$ and assume that $\ell \geq 3$. Let $(\A^{H_0},m^{H_0})$ be the Ziegler restriction onto $H_0$. Then it holds that 
\[b_2(\A) - (|\A|-1) \ge b_2 \left( \A^{H_0},m^{H_0} \right).\]
Moreover, 
$\A$ is free if and only if the above inequality is an equality, and 
$(\A^{H_0},m^{H_0})$ is free.
\end{theorem}

We restate two main results of~\cite{AK18} which will be used in the following proofs:

\begin{theorem}\cite[Theorem 1.2]{AK18}\label{th:heavy_thm}
		Let $(\A,m)$ be a multiarrangement and $H_0 \in \A$ a heavy hyperplane with $m_0\coloneqq m(H_0)$. Then
		\begin{equation}\label{eq:b2_heavy}
		b_2(\A,m)-m_0(|m|-m_0) \ge b_2(\A^{H_0},m^{H_0}).
		\end{equation}
		where $b_2\Am$ is the second Betti number of $\Am$ (cf.\ \cite[Definition 3.1]{ATW07}).
		Moreover, 
		$(\A,m)$ is free 
		if and only if the Euler--Ziegler restriction $(\A^{H_0},m^{H_0})$ of 
		$(\A,m)$ onto $H_0$ is free, and Inequality~\eqref{eq:b2_heavy} is satisfied with equality.
		In this case, 
		$\exp(\A,m)=(m_0,d_2,\ldots,d_\ell)$, where 
		$\exp(\A^{H_0},m^{H_0})=(d_2,\ldots,d_\ell)$.
\end{theorem}

\begin{theorem}\cite[Theorem 1.6]{AK18}\label{th:b2_ineq}
	Let $(\A,m)$ be a multiarrangement and $H \in \A$. Then 
	\[b_2(\A,m) - m(H)(|m|-m(H)) \ge b_2(\A^H,m^*),\] 
	where $(\A^H,m^*)$ is the Euler restriction of $(\A,m)$ onto $H$.
\end{theorem}

Lastly, we quote a lemma relating the logarithmic derivations of a multiarrangement with the ones of its restriction.
\begin{lemma}\cite[Lemma 3.2]{AK18}\label{lem:d0}
Let $\Am$ be a multiarrangement and fix a hyperplane $H_0\in \A$ with $H_0=\ker \alpha_{H_0}$.
If $\delta\in D\Am$ and $\delta(\alpha_{H_0})=0$ then $\delta\mid_{H_0} \in D(\A^{H_0},m^{H_0})$.
\end{lemma}

\section{Proofs of the Main Results}\label{sec:proof}

The starting point of this article is the observation that locally heavy hyperplanes guarantee the existence of a distinguished derivation in its module of logarithmic derivations playing a similar role as the Euler derivation for simple hyperplane arrangements.
Such a derivation is called a \textbf{good summand} in~\cite{AK18}.
This statement is made precise in the following proposition which generalizes Theorem 3.4 in~\cite{AK18} from heavy hyperplanes to locally heavy hyperplanes.

\begin{prop}\label{prop:good_summand}
Let $\Am$ be a multiarrangement with a locally heavy hyperplane $H_0\in\A$, say $H_0=\{\alpha_0=0\}$ and set $m_0 \coloneqq  m(H_0)$.
Further assume that $\Am$ is free with $\exp\Am =(d_1,\ldots,d_{\ell})$. Choose a basis $\theta_1\ldots,\theta_{\ell}$ of $D\Am$ with $d_i=\pdeg \theta_i$ for all $1\le i\le\ell$. Then the following three statements hold:
\begin{enumerate}
\item[(1)] For some $1\le j\le\ell$ it holds that $ d_j=m_0$ and $\theta_j(\alpha_0)=\alpha_0^{m_0}$.
\item[(2)] There exists a basis $\theta'_1,\dots,\theta'_{\ell}$ of $D\Am$ such that $\theta'_j(\alpha_0)=\alpha_0^{m_0}$ and the projections to $S/\alpha_0 S$ of the remaining basis elements $\overline{\theta'_1},\ldots,\overline{\theta'_{j-1}},\overline{\theta'_{j+1}},\ldots,\overline{\theta'_{\ell}}$ are $S/\alpha_0 S$-independent.
\item[(3)] The Euler restriction $(\A^{H_0},m^*)$ is free.
\end{enumerate}
\end{prop}
\begin{proof}
A change of coordinates allows us to assume $\alpha_0=x_1$.
By definition of $D\Am$ we have $\theta_i(x_1)=0$ for all $i$ with $d_i<m_0$.
Theorem~\ref{thm:saito} (Saito's criterion) ensures that there exists $1\le j\le \ell$ such that $\theta_j(x_1)\neq 0$.

Assume for a contradiction $d_i>m_0$ for all indices $i$ with $\theta_i(x_1)\neq 0$ and hence $\theta_i(x_1) = 0$ for all indices $i$ with $d_i\le m_0$.

Again by definition of $D\Am$ we can choose polynomials $f_i\in S$ such that $\theta_i(x_1)= x_1^{m_0} f_i$ for $1\le i \le \ell$.
Hence, the Laplace expansion of $M(\theta_1,\dots,\theta_\ell)$ along its first row yields
\[
\det M(\theta_1,\dots,\theta_\ell) = x_1^{m_0}\left( \sum_{i=1}^{\ell} (-1)^{i+1} f_ig_i\right),
\]
where $g_i\in S$ is the minor of $M(\theta_1,\dots,\theta_\ell)$ with the first row and $i$-th column removed. 
Therefore, Theorem~\ref{thm:saito} shows that $Q\Am= x_1^{m_0}\left( \sum_{i=1}^{\ell} (-1)^{i+1} f_ig_i\right)$.

Since $x_1^{m_0+1}\nmid Q\Am$ there exists an index, say $\ell$, such that 
\begin{equation}\label{eq:fg}
\overline{f_\ell g_\ell}\neq \overline{0}
\end{equation}
in $S/x_1S$.
In particular, we have
\[
\overline{0}\neq \overline{g_\ell}=\det(\overline{\theta_i(x_k)})_{i=1,\dots,\ell-1, k=2,\dots,\ell}.
\]
Hence, the derivations $\overline{\theta_1},\dots,\overline{\theta_{\ell-1}}$ are $S/x_1S$-independent.

Furthermore, \eqref{eq:fg} implies $f_\ell\neq 0$ and hence $d_\ell > m_0$ by assumption.
Therefore,
\begin{equation}\label{eq:deg_gl}
\deg g_\ell = |m|- d_\ell < |m|-m_0 = |m^*|,
\end{equation}
where the last equality holds since $H_0$ is locally heavy in $\Am$. 

By Proposition 2.2 in \cite{ATW08} the projected derivations $\overline{\theta_1},\dots,\overline{\theta_{\ell-1}}$ are elements of $D(\A^{H_0},m^*)$.
Hence, Theorem~\ref{thm:saito} implies that there exists some $f\in S$ such that
\[
\overline{g_\ell} = \overline{f}Q(\A^{H_0},m^*).
\]
Thus, $\deg g_\ell \ge |m^*|$ which contradicts~\eqref{eq:deg_gl}.
Therefore our assumption is false and we may without loss of generality assume $\theta_\ell(x_1)=x_1^{m_0}$ which proves the claim (1) of Proposition~\ref{prop:good_summand}.

For (2) we may assume $\theta_\ell(x_1)=x_1^{m_0}$ by the first part.
We perform a change of basis of $D\Am$ by setting $\theta'_i \coloneqq \theta_i - \frac{\theta_i(x_1)}{\theta_{\ell}(x_1)}\theta_{\ell}$ for $1\le i \le \ell -1$ and $\theta'_{\ell}\coloneqq \theta_{\ell}$.
Hence, the family $\theta'_1,\dots,\theta'_{\ell}$ also forms a basis of $D\Am$ and it holds that $\theta'_i(x_1)=0$ for $1\le i \le \ell -1$.
Thus, the Laplace expansion of the first row of the Saito matrix $M(\theta'_1,\dots,\theta'_{\ell})$ yields that the projected derivations $\overline{\theta'_1},\dots,\overline{\theta'_{\ell-1}}$ are $S/x_1S$-independent in $D(\A^{H_0},m^*)$.

Lastly, for (3) note that
\[
d_1+\dots + d_{\ell-1}=|m|-d_\ell=|m|-m_0=|m^*|.
\]
Thus, Saito's criterion (Theorem~\ref{thm:saito}) shows that the family $\overline{\theta'_1},\dots,\overline{\theta'_{\ell-1}}$ forms a basis of $D(\A^{H_0},m^*)$ which completes the proof.
\end{proof}

As a corollary to this proposition we obtain a strengthening of Proposition 3.6 in~\cite{AK18}:

\begin{cor}\label{cor:free}
	Let $\Am$ be a multiarrangement with $H_0\in \A$ such that $H_0$ is locally heavy. Then the following are equivalent:
	\begin{enumerate}
	\item[(1)] $\Am$ is free.
	\item[(2)] $(\A,m+k\delta_{H_0})$ is free for some $k\in\Z$ such that $H_0$ is also locally heavy in $(\A,m+k\delta_{H_0})$.
	\item[(3)] $(\A,m+k\delta_{H_0})$ is free for all $k\in\Z$ such that $H_0$ is locally heavy in $(\A,m+k\delta_{H_0})$.
	\end{enumerate}
	Note that $k$ is allowed to be negative in (2) and (3).
\end{cor}

\begin{proof}
Assume that $\Am$ is free. 
Proposition~\ref{prop:good_summand} then implies that $\exp \Am = (m_0,d_2,\ldots d_\ell)$ holds and $(\A^{H_0},m^{H_0})$ is also free with $\exp (\A^{H_0},m^{H_0}) = (d_2,\ldots d_\ell)$. 
This Euler--Ziegler restriction 
$(\A^{H_0},m^{H_0})$ coincides with the Euler
restriction $(\A^{H_0},m^*)$ on $\A^{H_0}$ since $H_0$ is a locally heavy hyperplane. 
So repeated applications of the addition in Theorem~\ref{thm:add_del} yields the freeness of $(\A,m+k\delta_H)$ since this arrangement and all intermediate arrangements are locally heavy by assumption with the identical Euler--Ziegler restriction. Conversely, if one assumes that $(\A,m+k\delta_H)$ is free, the same argument with the deletion part of Theorem~\ref{thm:add_del} shows the freeness of $\Am$.
\end{proof}

Before proving Theorem~\ref{th:loc_heavy} we give another Lemma:

\begin{lemma}\label{lem:b2}
Let $\Am$ be a multiarrangement and $H_0\in\A$ a locally heavy hyperplane with $m_0\coloneqq m(H_0)$.
Then
\[
b_2\Am -m_0(|m|-m_0) = \sum_{\substack{X\in L_2(\A), \\ X \not \subset H_0}} b_2(\A_X,m_X).
\]
\end{lemma}
\begin{proof}
The local-global formula (Theorem 3.3 in~\cite{ATW07}) enables us to express the second Betti number of $\Am$ as
\begin{align}
b_2\Am &= \sum_{X\in L_2(\A)} b_2(\A_X,m_X) \nonumber\\ 
&= \sum_{\substack{X\in L_2(\A), \\ X  \subset H_0}} b_2(\A_X,m_X)+ \sum_{\substack{X\in L_2(\A), \\ X \not \subset H_0}} b_2(\A_X,m_X).\label{eq:local_global}
\end{align}
Hence, it suffices to prove that the first summand in~\eqref{eq:local_global} equals $m_0(|m|-m_0)$.

In each localization $(\A_X,m_X)$ with $X\in L_2(\A)$ and $X \subset H_0$ the hyperplane $H_0$ is heavy by the definition of locally heavy.
Thus, these localizations are free with $\exp (\A_X,m_X) = (m_0,|m_X|-m_0)$, cf.\ e.g.\ \cite[Proposition 1.23 (i)]{Yos14}, and therefore we obtain $b_2(\A_X,m_X) = m_0(|m_X|-m_0)$.
Hence, we can compute
\begin{align}
 \sum_{\substack{X\in L_2(\A), \\ X  \subset H_0}} b_2(\A_X,m_X) &= \sum_{\substack{X\in L_2(\A), \\ X  \subset H_0}} m_0(|m_X|-m_0)\nonumber \\
 &= m_0\left( \sum_{\substack{X\in L_2(\A), \\ X  \subset H_0}} \sum_{\substack{K\in \A_X, \\ K\neq H_0}} m(K) \right)\label{eq:m_0}\\
 &=m_0(|m|-m_0).\nonumber
\end{align}
The double sum in Equation~\eqref{eq:m_0} counts the multiplicity of each hyperplane in $\A$ different from $H_0$ exactly once since each such hyperplane intersects $H_0$ exactly once.
Thus, this double sum equals $|m|-m_0$ which completes the proof.
\end{proof}

Combining the results derived so far, we are now able to prove Theorem~\ref{th:loc_heavy}:

\begin{proof}[Proof of Theorem~\ref{th:loc_heavy}]
By the assumption that $H_0$ is locally heavy in $\Am$, the Euler restriction $(\A^{H_0},m^*)$ coincides with the Euler--Ziegler restriction $(\A^{H_0},m^{H_0})$.
Therefore, the inequality~\eqref{eq:b2} follows immediately from Theorem~\ref{th:b2_ineq}.

For the characterization of the freeness of $\Am$ in the second part of Theorem~\ref{th:loc_heavy}, we fix some $k \in\N$ such that $H_0$ is a heavy hyperplane in $(\A,m+k\delta_{H_0})$.
By Corollary~\ref{cor:free} $\Am$ is free if and only if $(\A,m+k\delta_{H_0})$ is free.
Lemma~\ref{lem:b2} implies that the left hand side of Inequality~\eqref{eq:b2}, namely $b_2\Am -m_0(|m|-m_0)$, agrees for $\Am$ and $(\A,m+k\delta_{H_0})$.
Lastly by its definition, the Euler--Ziegler restriction $(\A^{H_0},m^{H_0})$ is independent of the multiplicity of $H_0$ and hence it agrees for $\Am$ and $(\A,m+k\delta_{H_0})$ too.
So in total, the characterization of freeness for $\Am$ follows immediately from the analogous characterization for $(\A,m+k\delta_{H_0})$ in Theorem~\ref{th:heavy_thm}.
\end{proof}

The proof of Theorem~\ref{th:flag} is very similar to the corresponding Theorem 1.7 in~\cite{AK18}.
We give it in its modified form for the sake of completeness.
\begin{proof}[Proof of Theorem~\ref{th:flag}]
For convenience in our notations, we set $\A=\left(\A^{X_0},m^{X_0}\right)$, where $m^{X_0} (H) =1$ for all $H\in\A$. Theorem~\ref{thm:AY} applied to $\A$ yields
\begin{equation}\label{Eq18}
\A \mbox{ is free} \Leftrightarrow \left( \A^{X_1},m^{X_1} \right) \mbox{ is free and } b_2(\A) -\left( \left| m^{X_0} \right|-1 \right) = b_2 \left(\A^{X_1},m^{X_1}\right).
\end{equation}
Note that $\left(\left(\A^{X_i}\right)^{X_{i+1}},\left(m^{X_i}\right)^{X_{i+1}}\right) = \left(\A^{X_{i+1}},m^{X_{i+1}}\right)$ and the hyperplane $X_{i+2}$ is locally heavy in the multiarrangement $\left(\A^{X_{i+1}},m^{X_{i+1}}\right)$ for $i=1,\ldots,\ell-3$.
In accordance with~\cite{AK18}, we will write $b_2^H\Am=b_2\Am-m(H)(|m|-m(H))$ to simplify our notation in the following.
Hence we can apply Theorem~\ref{th:loc_heavy} to obtain, 
for $i=1,\ldots,\ell-3$,
\begin{align}
\label{Eq19} &\left(\A^{X_{i}},m^{X_{i}}\right) \mbox{ is free} \Leftrightarrow \\
\nonumber &\left(\A^{X_{i+1}},m^{X_{i+1}} \right)\mbox{ is free and } b_2^{X_{i+1}} \left( \A^{X_{i}},m^{X_{i}} \right) = b_2 \left(\A^{X_{i+1}},m^{X_{i+1}}\right).
\end{align}
Since $\left(\A^{X_{\ell -2}},m^{X_{\ell -2}}\right)$ is a multiarrangement of rank 2, it is always free (cf.~\cite{Zie89}). So we can iteratively link the statements~\eqref{Eq18} and~\eqref{Eq19} to obtain
\begin{equation}\label{Eq20}
\A \mbox{ is free} \Leftrightarrow b_2^{X_{i+1}} \left( \A^{X_{i}},m^{X_{i}} \right) = b_2 \left(\A^{X_{i+1}},m^{X_{i+1}}\right) \mbox{for all } i=0,\ldots ,\ell -3.
\end{equation}
Since Theorems~\ref{thm:AY} and~\ref{th:loc_heavy} imply that $b_2^{X_{i+1}} \left( \A^{X_{i}},m^{X_{i}} \right) \ge b_2 \left(\A^{X_{i+1}},m^{X_{i+1}}\right)$ for all $i=0,\ldots ,\ell -3$, the right-hand side of~\eqref{Eq20} is equivalent to
\begin{equation}\label{eq100}
\sum_{i=0}^{\ell -3} b_2^{X_{i+1}} \left( \A^{X_{i}},m^{X_{i}} \right) = \sum_{i=0}^{\ell -3} b_2 \left(\A^{X_{i+1}},m^{X_{i+1}}\right).
\end{equation}
Noting that 
\begin{eqnarray*}
	b_2^{X_{i+1}}(\A^{X_{i}},m^{X_{i}})&=&
	b_2(\A^{X_i},m^{X_i})-m^{X_i}(X_{i+1})(|m^{X_i}|-m^{X_i}(X_{i+1}))
\end{eqnarray*}
we find that~\eqref{eq100} is in fact equivalent to
\begin{equation}\label{Eq21}
b_2 (\A) - \sum_{i=0}^{\ell -3} m^{X_{i}}(X_{i+1}) \left( \left|m^{X_{i}}\right|-m^{X_{i}}(X_{i+1} )\right) =b_2 \left(\A^{X_{\ell -2}},m^{X_{\ell-2}}\right) .
\end{equation}
Since $\exp \left(\A^{X_{\ell -2}},m^{X_{\ell -2}}\right)=\left( m^{X_{\ell -2}}(X_{\ell-1}), \left|m^{X_{\ell-2}}\right|-m^{X_{\ell-2}}(X_{\ell-1} ) \right)$, we have 
\[b_2 \left(\A^{X_{\ell -2}},m^{X_{\ell -2}}\right)= m^{X_{\ell -2}}(X_{\ell-1}) \left( \left|m^{X_{\ell-2}}\right|-m^{X_{\ell-2}}(X_{\ell-1} )\right).\]
This shows that~\eqref{Eq21} is equivalent to
\[
b_2 (\A) = \sum_{i=0}^{\ell -2} m^{X_{i}}(X_{i+1}) \left( \left|m^{X_{i}}\right|-m^{X_{i}}(X_{i+1} )\right)  .
\]
Since $\left|m^{X_{i}}\right|=\sum_{j=i}^{\ell-1} m^{X_{j}} (X_{j+1})$, the above 
equality is equivalent to
\begin{align*}
b_2 (\A) &= \sum_{i=0}^{\ell -2} m^{X_{i}}(X_{i+1}) \left( \sum_{j=i+1}^{\ell-1} m^{X_{j}} (X_{j+1})  )\right) \\
& = \sum_{0 \le i< j \le \ell-1}m^{X_i}(X_{i+1})m^{X_j}(X_{j+1})
\end{align*}
This completes the first part of the proof, since by~\eqref{Eq20} this is equivalent to $\A$ being free.

The characteristic polynomial $\chi (\A;t)$, and in particular also $b_2(\A)$, of a simple arrangement $\A$ is combinatorially determined.
The same holds true for the multiplicities of the Euler-Ziegler restrictions. 
Therefore the freeness of any arrangement in $LHF_\ell$ depends only on its combinatorics.
For the supersolvablity, use Proposition 4.2 in \cite{Abe17}.
\end{proof}

\section{Totally Non-Freeness and Generic Hyperplanes}\label{sec:generic}

Before stating the main result of this section we make the following definitions:
\begin{definition}
An arrangement $\A$ in a vector space $V$ is called \textbf{reducible} if after a change of coordinates $\A =\A_1\times \A_2$ where $\A_i$ is an arrangements of hyperplanes in the vector space $V_i$ for $i=1,2$ such that $V_1\oplus V_2=V$.
Otherwise, $\A$ is called \textbf{irreducible}.
A hyperplane $H$ in an arrangement $\A$ is defined to be a \textbf{generic hyperplane} if $|\A_X|=2$ for all $X\in L_2(\A)$ with $X\subset H$.
If all hyperplanes in $\A$ are generic the arrangement $\A$ is defined to be a \textbf{generic arrangement}.
\end{definition}

\begin{definition}\cite[Definition 5.4]{ATW08}
An arrangement $\A$ is called \textbf{totally nonfree} if for any multiplicity $m: \A \rightarrow \Z_{>0}$ the multiarrangement $\Am$ is not free.
\end{definition}

Yoshinaga showed that an irreducible generic arrangement of rank greater than $2$ is totally nonfree~\cite{Yos10}.
A simple irreducible arrangement of rank greater than $2$ which contains a generic hyperplane is known to be not free, cf.\ \cite{OT92}.
Recently, DiPasquale generalized this fact to multiarrangements by showing that such an arrangement is totally nonfree \cite[Corollary 4.13]{DiP18}.
Since a generic hyperplane is always locally heavy, we obtain another proof of this result as a corollary of our locally heaviness technique.

\begin{cor}\label{cor:generic}
Let $\A$ be an irreducible arrangement of rank greater than $2$ with a generic hyperplane $H_0\in \A$. Then $\A$ is totally nonfree.
\end{cor}

Before proving this corollary, we recall a known auxiliary lemma:
\begin{lemma}\label{lem:degree}
Let $\Am$ be an irreducible multiarrangement with $m\not\equiv 1$. Then for any non-zero $\theta\in D\Am$ it holds that $\pdeg \theta \ge 2$.
\end{lemma}
\begin{proof}
For a graded $S$-module $M$ we denote by $M_k$ the $k$-th graded piece of $M$.
We will firstly show that $D(\A)_0=0$ and $D(\A)=\langle \theta_E\rangle_S$ where $\theta_E=\sum_{i=1}^{\ell}x_i\p_{x_i}$ is the Euler derivation.

Assume there is a $\theta\in D(\A)_0$.
Hence, $\theta=\sum_{i=0}^{\ell}c_i\p_{x_i}$ for some $c_i\in \K$ for $i=1,\dots,\ell$ and we set $c\coloneqq (c_1,\dots,c_{\ell})$.
Therefore, for any $H\in\A$ we have $0=\theta(\alpha_H)=c\cdot \alpha_H$, which contradicts $\A$ being irreducible.

Secondly, consider a $\theta\in D(\A)_1$ with $\theta\notin \langle \theta_E\rangle_S$ and write again $\theta=\sum_{i=0}^{\ell}f_i\p_{x_i}$ for some $f_i\in S_1$ for $i=1,\dots,\ell$.
Since $\A$ is irreducible by assumption we may assume that $\A$ contains the hyperplanes $H_i\coloneqq \{x_i=0\}$ for all $1\le i \le \ell$.
Since $\theta(x_i)=f_i$ we find that $f_i=c_ix_i$ for suitable $c_i\in \K$ for all $1\le i \le \ell$.
By subtracting the scaled Euler derivation $c_1\theta_E$ from $\theta$ we can assume $c_1=0$, i.e.\ $\theta(x_1)=0$ and $\theta\neq 0$ since $\theta$ is not a multiple of $\theta_E$. 
This however implies that the defining linear equations of $\A$ except for $H_1$ only involve the variables $x_2,\dots,x_\ell$ since any hyperplane in both $x_1$ and any other variable would contradict $\theta\in D(\A)$.
This is a contradiction to the fact that $\A$ is irreducible which finishes the proof of our claim.

Now finally, we have $D(\A)_k\supseteq D\Am_k$ by definition of the derivation module.
By assumption, there is $H\in\A$ with $m(H)>1$.
Hence, $\theta_E(\alpha_H)=\alpha_H\notin \alpha_H^{m(H)}S$ which implies $\theta_E\notin D\Am$.
In light of the first part of the proof, we obtain $D\Am_0=0$ and $D\Am_1=0$.
\end{proof}

\begin{proof}[Proof of Corollary~\ref{cor:generic}]
Assume $\Am$ is free for some multiplicity $m$ with $\exp\Am=(d_1,\dots,d_\ell)$.
Firstly, we consider the case $m_0\coloneqq m(H_0)=1$.
In the case of a simple arrangement, i.e. $m(H)=1$ for all $H\in\A$, $\Am$ is known to be not free.
So we may assume $m\not\equiv 1$.
By assumption, the hyperplane $H_0$ is generic in $\A$ which implies for the Euler restriction $(\A^{H_0},m^*)$ of $\Am$ onto $H_0$
\[
	|m^*|=|m|-m_0=|m|-1.
\]
The fact that the arrangement $\A$ is irreducible and $m\not\equiv 1$ yields $d_i\ge 2$ for all $1\le i \le \ell$ by Lemma~\ref{lem:degree}.
Furthermore, by Proposition~\ref{prop:good_summand} at least one of the $(\ell-1)$ subsets of a basis of $D\Am$ forms a basis of $D(\A^{H_0},m^*)$.
So by combining the above statements with Saito's criterion we obtain
\[
	|m|-1=|m^*| \le |m|-2,
\]
which is a contradiction and hence $\Am$ is not free.

As a second case, we assume $m_0 >1$.
Since $H_0$ is a generic hyperplane it is by definition also locally heavy in the multiarrangement $\Am$.
Hence, Corollary~\ref{cor:free} implies that $\Am$ is free if and only if $(\A,\tilde{m})$ is free where the multiplicity $\tilde{m}$ is defined as
\[\tilde{m}(H)\coloneqq 
\begin{cases}
1& \mbox{if }H=H_0,\\
m(H)&\mbox{otherwise.}
\end{cases}
\]
However, the multiarrangement $(\A,\tilde{m})$ is not free by the first case which completes the proof.
\end{proof}

\section{Multiple Locally Heavy Hyperplanes}\label{sec:multiple}

In contrast to heavy hyperplanes a multiarrangement can have multiple locally heavy hyperplanes.
The following results show that in this case the multiarrangement is not free unless its combinatorics is of a special nature.

\begin{theorem}\label{th:twolh}
Let $\Am$ be a multiarrangement with two locally heavy hyperplanes $H,L\in \A$.
Then we have:
\begin{enumerate}[(1)]
\item If $\Am$ has rank $3$, it is free if and only if it is reducible.
\item Assume $\Am$ has rank $\ell> 3$, and there exists an $X\in L_3(\A)$ with $X\subset H\cap L$ such that $\A_X$ is reducible as $\A_X = \emptyset_{\ell-3}\times \widetilde{\A_X}$ where $\emptyset_{d}$ is the empty arrangement in a $\K^d$ and $\widetilde{\A_X}$ is irreducible.
Then $\Am$ is not free.
\end{enumerate}
\end{theorem}
\begin{proof}
For (1) note that a reducible rank $3$ multiarrangement is always free since any rank $2$ multiarrangement is free.
So let $\Am$ be irreducible of rank $3$ and assume it is free.
Then by Proposition~\ref{prop:good_summand} we have $\exp\Am=(m(H),m(L),d_3)$. 
Theorem~\ref{th:loc_heavy} now yields that $(\A^H,m^H)$ is free with $\exp(\A^H,m^H) =(m(L),d_3)$ and hence $b_2(\A^H,m^H) =m(L)d_3$.
On the other hand, Theorem~\ref{th:loc_heavy} also shows
\begin{equation}\label{eq:contradiction}
b_2\Am -m(H)(|m|-m(H)) = b_2(\A^H,m^H)=m(L)d_3.
\end{equation}
By Lemma~\ref{lem:b2}, the left hand side of this equation can computed as
\begin{align*}
b_2\Am -m(H)(|m|-m(H))&=\sum_{\substack{X\in L_2(\A), \\ X \not \subset H}} b_2(\A_X,m_X)\\
&= \sum_{\substack{X\in L_2(\A), \\ X \not \subset H\\X  \subset L}} b_2(\A_X,m_X) + \sum_{\substack{X\in L_2(\A), \\ X \not \subset H\\X  \not \subset L}} b_2(\A_X,m_X)\\
&=m(L)d_3+ \sum_{\substack{X\in L_2(\A), \\ X \not \subset H\\X  \not \subset L}} b_2(\A_X,m_X),
\end{align*}
where the last equation holds since $L$ is locally heavy and it intersects all remaining hyperplanes whose multiplicities sum to $d_3$.
Since $\Am$ is irreducible there is at least one $X\in L_2(\A)$ with $ X \not \subset H$ and $X  \not \subset L$.
Hence, the second summand in the above equation is strictly positive which is a contradiction to Equation~\eqref{eq:contradiction}.
Therefore, $\Am$ is not free in this case.

For (2) assume again that $\Am$ is free.
Now consider the localization $(\A_X,m_X)$ which is also free since any localization of a free is multiarrangement is free.
The rank $3$ multiarrangement $(\A_X,m_X)$ contains the hyperplanes $H,L$ which are clearly also locally heavy in $(\A_X,m_X)$.
Furthermore, $(\A_X,m_X)$ contains the irreducible component $(\widetilde{\A_X},m_X)$ by assumption.
Hence, by part (1) the multiarrangement $(\A_X,m_X)$ is not free which is a contradiction.
Thus, $\Am$ is not free.
\end{proof}

\begin{example}
	Consider again a multiarrangement $\Am$ whose underlying simple arrangement is the braid arrangement $A_3$ and which is given by the defining equation
	\[
	\Q\Am = x(x-y)^2(x-z)y(y-z)z^2.
	\]
	Both the hyperplanes $\{x-y=0\}$ and $\{z=0\}$ are locally heavy in $\Am$.
	Since the underlying braid arrangement is irreducible, Theorem~\ref{th:twolh} immediately shows that $\Am$ is nonfree.
\end{example}

\printbibliography

\end{document}